\documentclass[12pt]{amsart}

\usepackage{amscd}
\usepackage{amsmath}
\usepackage{amssymb}
\usepackage{amsthm}
\usepackage{epsf}
\usepackage{latexsym}
\usepackage{verbatim}
\usepackage[dvips]{graphicx}
\input epsf.tex
\usepackage[all, cmtip]{xy}

\newtheorem{definition}{Definition}
\newtheorem{question}{Question}
\newtheorem{theorem}{Theorem}[section]

\newtheorem{corollary}[theorem]{Corollary}
\newtheorem{proposition}[theorem]{Proposition}
\newtheorem{varexample}[theorem]{Example}

\newcommand{\PP}{\mathbb{P}\,}
\newcommand{\LL}{\mathcal{L}\,}
\newcommand{\OO}{\mathcal{O}\,}

\newcommand{\Z}{\mathbb{Z}\,}

\begin{document}
\title{Birational Contractions of $\overline{M}_{3,1}$ and $\overline{M}_{4,1}$}
\author{David Jensen}
\date{}
\bibliographystyle{alpha}
\maketitle

\begin{abstract}

We study the birational geometry of $\overline{M}_{3,1}$ and $\overline{M}_{4,1}$.  In particular, we pose a pointed analogue of the Slope Conjecture and prove it in these low-genus cases.  Using variation of GIT, we construct birational contractions of these spaces in which certain divisors of interest -- the pointed Brill-Noether divisors -- are contracted.  As a consequence, we see that these pointed Brill-Noether divisors generate extremal rays of the effective cones for these spaces.

\end{abstract}

\section{Introduction}

The moduli spaces of curves are some of the most studied objects in algebraic geometry.  In recent years, a great deal of progress has been made on understanding the birational geometry of these spaces.  Examples include the work of Hassett and Hyeon on the minimal model program for $\overline{M}_g$ \cite{HH1} \cite{HH2} and the discovery by Farkas of previously unknown effective divisors on $\overline{M}_g$ \cite{Farkas}.  Nevertheless, many fundamental questions remain open.

Many of these questions can be stated in terms of the cone of effective divisors $\overline{NE}^1 ( \overline{M}_g )$. Among the first to study this cone were Eisenbud, Harris and Mumford in a series of papers proving that $\overline{M}_g$ is of general type for $g \geq 24$ \cite{HMum} \cite{HE}.  A key element of these proofs is the computation of the class of certain divisors on $\overline{M}_g$.  The original paper of Harris and Mumford focused on the $k$-gonal divisor in $\overline{M}_{2k-1}$, a specific case of the more general class of Brill-Noether divisors.  In their argument, they use this calculation to show that the canonical class can be written as an effective sum of a Brill-Noether divisor, boundary divisors, and an ample divisor, and hence lies in the interior of $\overline{NE}^1 ( \overline{M}_g )$.  The search for effective divisors with this property eventually led to the Harris-Morrison Slope Conjecture.

In their work, Harris and Eisenbud discovered that all of the Brill-Noether divisors lie on a single ray in $\overline{NE}^1 ( \overline{M}_g )$.   One consequence of the Slope Conjecture would be that this ray is extremal.  The Slope Conjecture has recently been proven false in \cite{FP} and subsequently in \cite{Farkas}, but the statement is known to hold for certain small values of $g$.  In several of these cases, the statement can be proved by use of the Contraction Theorem, which states that the set of exceptional divisors of a birational contraction $X \dashrightarrow Y$ span a simplicial face of $\overline{NE}^1 (X)$ (see \cite{Rulla}). In other words, the Slope Conjecture has been shown to hold for small values of $g$ by constructing explicit birational models for the moduli space in which the Brill-Noether divisor is contracted.  Moreover, these models arise naturally as geometric invariant theory quotients.

The purpose of this paper is to carry out a pointed analogue of the discussion above in some low genus cases. In \cite{Logan}, Logan introduced the notion of \textbf{pointed Brill-Noether divisors}.

\begin{definition}
Let $Z = (a_0 , \ldots , a_r )$ be an increasing sequence of nonnegative integers with $\alpha = \sum_{i=0}^r (a_i - i)$.  Let $BN^r_{d,Z}$ be the closure of the locus of pointed curves $(p,C) \in M_{g,1}$ possessing a $g^r_d$ on $C$ with vanishing sequence $Z$ at $p$.  When $g+1 = (r+1)(g-d+r)+ \alpha$, this is a divisor in $\overline{M}_{g,1}$, called a \textbf{pointed Brill-Noether divisor}.
\end{definition}

Logan's original motivation was to prove a pointed version of the Harris-Mumford general type result.  In this setting, it is natural to consider an analogue of the Slope Conjecture:

\begin{question}
Is there an extremal ray of $\overline{NE}^1 ( \overline{M}_{g,1} )$ generated by a pointed Brill-Noether divisor?
\end{question}

We consider this question in certain low-genus cases.  When $g=2$, this question was answered in the affirmative by Rulla \cite{Rulla}.  He shows that the Weierstrass divisor $BN^1_{2,(0,2)}$ generates an extremal ray of $\overline{NE}^1 ( \overline{M}_{2,1} )$ by explicitly constructing a birational contraction of $\overline{M}_{2,1}$.  Our main result is an extension of this to higher genera:

\begin{theorem}

There is a birational contraction of $\overline{M}_{3,1}$ contracting the Weierstrass divisor $BN^1_{3,(0,3)}$.  Similarly, there is a birational contraction of $\overline{M}_{4,1}$ contracting the pointed Brill-Noether divisor $BN^1_{3,(0,2)}$.

\end{theorem}

As a consequence, we identify an extremal ray of the effective cone.

\begin{corollary}

For $g = 3,4$, there is an extremal ray of $\overline{NE}^1 ( \overline{M}_{g,1} )$ generated by a pointed Brill-Noether divisor.

\end{corollary}

The proof uses variation of GIT.  In particular, we consider the following GIT problem:  let $Y$ be a surface and fix a linear equivalence class $\vert D \vert$ of curves on $Y$.  Now, let
$$ X = \{ (p,C) \in Y \times \vert D \vert \text{  } \vert p \in C \} $$
be the universal family over this space of curves.  In the case where $(Y, \vert D \vert )$ is $( \PP^2 , \vert \OO (4) \vert )$ or $( \PP^1 \times \PP^1 , \vert \OO (3,3) \vert )$, the quotient of $X//Aut(Y)$ is a birational model for $\overline{M}_{3,1}$ or $\overline{M}_{4,1}$, respectively.  By varying the choice of linearization, we obtain a birational model in which the specified divisor is contracted.

The outline of the paper is as follows.  In section 2 we provide some background on variation of GIT.  In section 3, we develop a tool for studying GIT quotients of families of curves on surfaces.  In particular, we construct a large class of divisors on these spaces that are invariant under the automorphism group of the surface, called Hessians.  In sections 4 and 5 we then examine separately curves on $\PP^2$ and on $\PP^1 \times \PP^1$, yielding our result in the cases of $g = 3$ and 4.

We plan on discussing similar results for genus 5 and 6 in a later paper.

\textbf{Acknowledgements}
This work was prepared as part of my doctoral dissertation under the direction of Sean Keel.  I would like to thank him for his abundance of help and suggestions.  I would also like to thank Brendan Hassett for his ideas.

\section{Variation of GIT}

The birational contractions that we construct arise naturally as GIT quotients.  This section contains a brief summary of results of Dolgachev-Hu \cite{DH} and Thaddeus \cite{Thaddeus} on variation of GIT.

Given a group $G$ acting on a variety $X$, the GIT quotient $X//G$ is not unique -- it depends on the choice of a $G$-ample line bundle.  In particular, if $\LL \in Pic^G (X)$, we have
$$ X //_{ \LL } G = Proj \bigoplus_{n \geq 0} H^0 (X, \LL^{\otimes n} )^G . $$
Following Dolgachev and Hu, we will call the set of all $G$-ample line bundles the \textbf{$G$-ample cone}.  A study of how the quotient varies with the choice of the $G$-ample line bundle was carried out independently by Dolgachev-Hu \cite{DH} and Thaddeus \cite{Thaddeus}.  The following theorem is a summary of some of the results of those papers:

\begin{theorem} \cite{DH} \cite{Thaddeus}
The $G$-ample cone is divided into a finite number of convex cones, called \textbf{chambers}.  Two line bundles $\LL$ and $\LL '$ lie in the same chamber if $X^s (\LL ) = X^{ss} (\LL ) = X^{ss} (\LL ') = X^s (\LL ')$.  The chambers are bounded by a finite number of \textbf{walls}.  A line bundle $\LL$ lies on a wall if $X^{ss}(\LL ) \neq X^s (\LL )$.  If $\LL$ lies on a wall and $\LL '$ lies is an adjacent chamber, then there is a morphism $X//_{\LL'}G \to X//_{\LL}G$.  This map is an isomorphism over the stable locus.
\end{theorem}

Both Thaddeus and Dolgachev-Hu examine the maps between quotients at a wall in the $G$-ample cone.  Specifically, let $\LL_+$, $\LL_-$ be $G$-ample line bundles in adjacent chambers of the $G$-ample cone, and define $\LL (t) = \LL_+^t \otimes \LL_-^{1-t}$.  Suppose that the line between them crosses a wall precisely at $\LL (t_0 )$.  Following Thaddeus, define
$$X^{\pm} = X^{ss} (\LL_{t_0}) \backslash X^{ss} (\LL_{\mp} )$$
$$X^0 = X^{ss} (\LL_{t_0}) \backslash (X^{ss} (\LL_+) \cup X^{ss} (\LL_-))$$

\begin{theorem}
\label{Thaddeus}

\cite{Thaddeus}  Let $x \in X^0$ be a smooth point of $X$.  Suppose that $G \cdot x$ is closed in $X^{ss}(\LL_{t_0})$ and that $G_x \cong \mathbb{C}^*$.  Then the natural map $X //_{\LL_{ \pm }} G \to X //_{\LL_{t_0}} G$ is an isomorphism outside of $X^{\pm}//_{\LL_{\pm}} G$.  Over a neighborhood of $x$ in $X^0 //_{t_0} G$, $X^{\pm} //_{\LL_{\pm}} G $ are fibrations whose fibers are weighted projective spaces.

\end{theorem}

In order to determine whether a point is (semi)stable, we will make frequent use of Mumford's numerical criterion.  Given a $G$-ample line bundle $\LL$ and a one-parameter subgroup $\lambda : \mathbb{C}^* \to G$, it is standard to choose coordinates so that $\lambda$ acts diagonally on $H^0 (X, \LL )^*$.  In other words, it is given by $diag(t^{a_1}, t^{a_2},  \ldots , t^{a_n} )$.  We will refer to the $a_i$'s as the weights of the $\mathbb{C}^*$ action.  For a point $x \in X$, Mumford defines
$$\mu_{\lambda} (x) = min(a_i \vert x_i \neq 0).$$
Then $x$ is stable (semistable) if and only if $\mu_{\lambda} (x) < 0$ (resp. $\mu_{\lambda} (x) \leq 0$) for every nontrivial 1-parameter subgroup $\lambda$ of $G$ (see Theorem 2.1 in \cite{Mumford}).

\section{Hessians}

Here we set up the GIT problem that appears in sections 4 and 5.  We also identify a collection of $G$-invariant divisors that will be useful for analyzing this problem.

Let $Y$ be a smooth projective surface over $\mathbb{C}$, $\LL '$ an effective line bundle on $Y$, and $Z = \PP H^0(Y, \LL ')$.  Let $$X = \{ (p,C) \in Y \times Z \vert p \in C \}.$$  We denote the various maps as in the following diagram:
$$\xymatrix{
X \ar[d]^f \ar[r]^i & Y \times Z \ar[d]^{\pi_2} \ar[r]^{\pi_1} & Y \\
Z \ar[r]^{id} &Z}$$
If $\LL'$ is base-point free, then $X$ is a projective space bundle over $Y$, so it is smooth and  $PicX \cong PicY \times \Z$.  We will later study the GIT quotients of $X$ by the natural action of $Aut(Y)$.

If $C$ is a curve on $Y$ and $\LL$ is another line bundle on $Y$, then for every point $p \in C$ there are $n+1 = h^0(C, \LL \vert_C )$ different orders of vanishing of sections $s \in H^0(C, \LL \vert_C )$.

\begin{definition}
When written in increasing order, $$ a_0^{\LL}(p) < \cdots < a_n^{\LL}(p)$$ the orders of vanishing are called the \textbf{vanishing sequence} of $\LL$ at $p$. The \textbf{weight} of $\LL$ at $p$ is defined to be $w^{\LL}(p) = \sum_{i=0}^n a_i^{\LL}(p) - i$.  A point is said to be an $\LL$-\textbf{flex} if the weight of $\LL$ at the point is nonzero.
\end{definition}

In other words, $p$ is an $\LL$-flex if the vanishing sequence of $\LL$ at $p$ is anything other than $0 < 1 < \cdots < n$.

\begin{definition}
The \textbf{divisor of $\LL$-flexes} is $\sum_{p \in C}w^{\LL}(p)p$.  It corresponds to a section $W_{\LL}$ of a certain line bundle called the \textbf{Wronskian} of $\LL$.  We say that a curve $H$ on $Y$ is an \textbf{$\LL$-Hessian} if the restriction of $H$ to $C$ is precisely the divisor of $\LL$-flexes.
\end{definition}

Returning to our family of curves $f : X \to Z$ above, suppose that $\LL$ is a line bundle on $Y$ such that the pushforward $f_* ( \pi_1 \circ i )^* \LL$ is locally free of rank $n+1$.  We define a relative $\LL$-Hessian to be a divisor $H \subseteq X$ whose restriction to each fiber is the divisor of $f_* ( \pi_1 \circ i )^* \LL$-flexes.  Relative $\LL$-Hessians were studied by Cukierman \cite{Cukierman}, who shows:

\begin{proposition}
\label{Cukierman}
\cite{Cukierman} The class of the relative $\LL$-Hessian is
$$ (n+1)c_1 ( \pi_1 \circ i)^* \LL + {{n+1}\choose{2}} c_1 \Omega^1_{X/Z} - c_1 f^* f_* ( \pi_1 \circ i )^* \LL .$$
\end{proposition}

In our particular case, we can determine this class more explicitly.

\begin{corollary}
For $X,Y,$ and $Z$ as above, the class of the relative $\LL$-Hessian is
$$ (n+1)c_1 ( \pi_1 \circ i)^* \LL + {{n+1}\choose{2}} ( c_1 \pi_1^* \Omega^1_Y \vert_X + c_1 ( \pi_1 \circ i)^* \LL ' + c_1 f^* \OO_Z (1))$$
$$ - h^0 (Y , \LL \otimes \LL '^* ) ( c_1 f^* \OO_Z (1) ).$$
\end{corollary}

\begin{proof}
We follow the proof in \cite{Cukierman}.  If $I$ is the ideal sheaf of $X$ in $Z \times Y$, then we have the exact sequence
$$ 0 \to I/I^2 \to \pi_1^* \Omega^1_Y \vert_X \to \Omega^1_{X/Z} \to 0 $$
so we have
$$ c_1 \Omega^1_{X/Z} = c_1 \pi_1^* \Omega^1_Y \vert_X - c_1 I/I^2.$$

Also, $X$ is the scheme of zeros of a section of the line bundle $E = ( \pi_1 \circ i)^* \LL ' \otimes f^* \OO_Z (1)$ on $Y \times Z$.  Note that $I/I^2 \cong E^* \otimes \OO_X = E^* \vert_X$.  It follows that
$$c_1 \Omega^1_{X/Z} = c_1 ( \pi_1 \circ i)^* \Omega^1_Y \vert_X + c_1 E $$
$$= c_1 ( \pi_1 \circ i)^* \Omega^1_Y \vert_X  + c_1 ( \pi_1 \circ i)^* \LL ' + c_1 f^* \OO_Z (1).$$

Now, consider the exact sequence on $Y \times Z$
$$ 0 \to \pi_1^* L \otimes E^* \to \pi_1^* L \to \pi_1^* L \vert_X \to 0 $$
From the projection formula, we see that
$$\pi_{2*}( \pi_1^* \LL \otimes E^* ) = H^0 (Y , \LL \otimes \LL '^* ) \otimes \OO_Z (-1)$$
and $R^1 \pi_{2*} (\pi_1^* L \otimes E^* ) = 0$.  This gives us the exact sequence on $Z$
$$ 0 \to \pi_{2*} ( \pi_1^* L \otimes E^* ) \to \pi_{2*} \pi_1^* L \to \pi_{2*} ( \pi_1^* L \vert_X ) \to 0 $$
Since the middle term is a trivial bundle, the result follows from Proposition \ref{Cukierman}.
\end{proof}

For the remainder of this section, we identify specific examples that will appear in the arguments to follow.

In section 4 we consider the case that $Y = \PP^2$ and $\LL' = \OO_Y (d)$ for some $d \geq 3$.  By the above, we see that for every $m$ and $d$, a relative $\OO_Y (m)$-Hessian $H_m$ exists.  Since $c_1 \pi_1^* \Omega^1_Y \vert_X = \OO_X (-3,0)$, if $m<d$, $H_m$ is cut out by a $G$-invariant section $W_m$ of $$\OO_X ((n+1)m + {{n+1}\choose{2}} (d-3), {{n+1}\choose{2}} ),$$ where $n+1 = h^0 (Y, \LL ) = {{m+2}\choose{2}}$.

In particular, $H_1$ is cut out by a section $W_1 \in H^0( \OO_X(3(d-2),3))$.  $W_1$ vanishes at $(p,C)$ if $C$ is smooth at $p$ and the tangent line to $C$ at $p$ intersects $C$ with multiplicity at least 3, or if $p$ is a singular point of $C$.  Similarly, $H_2$ is defined by a section of $W_2 \in H^0( \OO_X(15d-33, 15))$.  $W_2$ vanishes at $(p,C)$ if $C$ is smooth at $p$ and the osculating conic to $C$ at $p$ intersects $C$ with multiplicity at least 6, or if $p$ is a singular point of $C$.

It is known that $H_2 = H_1 \cup H_2'$ is reducible ( see Proposition 6.6 in \cite{CF}).  Indeed, if a line meets $C$ with multiplicity 3 at $p$, then the double line meets $C$ with multiplicity 6 at $p$.  The points of $H_2' \cap C$ are classically known as the \textbf{sextatic points} of $C$, and $H_2 '$ is cut out by a $G$-invariant section $W_2 '$ of $\OO_X (12(d-\frac{9}{4}),12)$.  A simple calculation shows that $H_2 ' \cap C$ also contains those points of $C$ where $w^{\OO_C (1)} (p) > 1$.  These include singular points and points where the tangent line to $C$ is a \textbf{hyperflex} (a line that intersects $C$ at $p$ with multiplicity $\geq 4$).

Similarly, in section 5 we consider the case that $Y = \PP^1 \times \PP^1$, and $\LL' = \OO_Y (d,d)$.  Note that, for every $(m_1 , m_2 , d)$ with $m_i < d$, a relative $\OO_Y (m_1 , m_2 )$-Hessian $H'_{m_1 ,m_2}$ exists. In this case, our formulas show that the rank of $f_* ( \pi_1 \circ i )^* \OO_Y (m_1 , m_2 )$ is
$$ n+1 = h^0 (\OO_Y (m_1 , m_2 )) = (m_1 + 1)(m_2 + 1). $$
Also, since $c_1 \pi_1^* \Omega^1_Y \vert_X = \OO_X (-2,-2,0)$, we see that $H'_{m_1 ,m_2}$ is cut out by a section $W'_{m_1 ,m_2} \in H^0 ( \OO_X (a_1 , a_2 , b))$ for
$$ a_i = (n+1)m_i + {{n+1}\choose{2}}(d-2) $$
$$ b = {{n+1}\choose{2}}.$$

 Since $\PP^1 \times \PP^1$ has a natural involution, we know that $W'_{m_1 , m_2}$ cannot be $G$-invariant if $m_1 \neq m_2$.  Notice, however, that $W'_{m_1 , m_2} \otimes W'_{m_2 , m_1}$ is a $G$-invariant section of $\OO_X (a, a, b)$ for
$$ n+1 = (m_1 + 1)(m_2 + 1) $$
$$ a = (n+1)(m_1 + m_2 ) + 2{{n+1}\choose{2}}(d-2) $$
$$ b = 2{{n+1}\choose{2}}.$$
We will use $W_{m_1 , m_2}$ to denote the $G$-invariant section described here, and $H_{m_1 ,m_2}$ to denote its zero locus.

In particular, $W_{0,1} \in H^0 (\OO_X (2(d-1),2(d-1),2))$.  It vanishes at a point $(p,C)$ if $C$ intersects one of the two lines through $p$ with multiplicity at least 2 (or, equivalently, if the osculating $(1,1)$ curve is a pair of lines).  Similarly, $W_{1,1} \in H^0 (\OO_X (2(3d-4),2(3d-4),6))$.  It vanishes at a point $(p,C)$ if there is a curve of bidegree $(1,1)$ that intersects $C$ with multiplicity 4 or more at $p$.

\section{Contraction of $\overline{M}_{3,1}$}

In this section, we prove our main result in the genus 3 case:

\begin{theorem}
There is a birational contraction of $\overline{M}_{3,1}$ contracting the Weierstrass divisor $BN^1_{3,(0,3)}$.
\end{theorem}

In order to construct a birational model for $\overline{M}_{3,1}$, we consider GIT quotients of the universal family over the space of plane quartics.  The image of the Weierstrass divisor in this model is precisely the Hessian $H_1$, and we exhibit a GIT quotient in which this locus is contracted.  For most of this section we will consider, more generally, plane curves of any degree $d \geq 3$.

Specifically, following the set-up of the previous section, we let
$$X = \{ (p,C) \in \PP^2 \times \vert \OO (d) \vert \text{  } \vert p \in C \}.$$
Then $\pi_2 : X \to \vert \OO (d) \vert$ is the family of all plane curves of degree $d$.  Our goal is to study the GIT quotients of $X$ by the action of $G = PSL(3,\mathbb{C})$.  By the above, we know that $PicX \cong \Z \times \Z$, so the quotient $X //_{\LL} G$ depends on a single parameter $t$ which we call the \textbf{slope} of $\LL$.

\begin{definition}

We say a line bundle $\LL$ has \textbf{slope} $t$ if $\LL = \pi_1^* \OO (a) \otimes \pi_2^* \OO (b)$ with $t = \frac{a}{b}$.  We write $X^s (t)$ and $X^{ss} (t)$ for the sets of stable and semistable points, and $X //_t G $ for the corresponding GIT quotient.

\end{definition}

Here we describe the numerical criterion for points in $X$.  Let $p = (x_0 , x_1 , x_2 )$ and
$$ C = \sum_{i+j+k=d} a_{i,j,k} x_0^i x_1^j x_2^k.$$
Then a basis for $H^0 ( \OO_X (a,b))$ consists of monomials of the form
$$ \prod_{\alpha = 1}^a x_{l_{ \alpha }} \prod_{\beta = 1}^b a_{i_{\beta},j_{\beta},k_{\beta}}. $$
The one-parameter subgroup with weights $(r_0 , r_1 , r_2 )$ acts on the monomial above with weight
$$ \sum_{\alpha =1}^a r_{l_{\alpha}} - \sum_{\beta = 1}^b ( i_{\beta} r_0 + j_{\beta} r_1 + k_{\beta} r_2 ). $$
In our case, we will only be interested in maximizing or minimizing this weight, so it suffices to consider monomials of the form $x_l^a a_{i,j,k}^b$.  In this case, the one-parameter subgroup acts with weight $ar_l - b(ir_0 + jr_1 + kr_2 )$, which is proportional to
$$ \mu_{\lambda} (x_l , a_{i,j,k} ) := tr_l - (ir_0 + jr_1 + kr_2 ). $$

The $G$-ample cone of $X$ has two edges, one of which occurs when $t=0$.  In the case where $d=4$, we obtain the well-known moduli space of plane quartics.  Descriptions of $X^s (0)$ and $X^{ss} (0)$ appear in \cite{Mumford}, and the quotient $X //_0 G$ plays an important role in the birational geometry of $\overline{M}_3$.  For example, Hyeon and Lee show that this quotient is a log canonical model for $\overline{M}_3$ \cite{HL}, and the space also appears in work on moduli of K3 surfaces \cite{Artebani} and cubic threefolds \cite{CML}.

We will see that, when $t$ is large, stability conditions reflect the inflectionary behavior of linear series at the marked point.  Thus, as $t$ increases, the curve is allowed to have more complicated singularities, but vanishing sequences at the marked point become more well-behaved.

Our first result is to identify the other edge of the $G$-ample cone.  It is determined by the Wronksian $W_1$.

\begin{proposition}

An edge of the $G$-ample cone occurs at $t = d-2$.

\end{proposition}

\begin{proof}

It suffices to show that $X^{ss}(d-2) \neq X^s (d-2) = \emptyset$.  It is clear that $X^{ss}(d-2) \neq \emptyset$, since $W_1$ is a $G$-invariant section of $\OO_{X} (3(d-2),3)$.

To show that $X^s (d-2) = \emptyset$, we invoke the numerical criterion.  Let $(p,C) \in X$.  By change of coordinates, we may assume that $p = (0,0,1)$ and the tangent line to $C$ at $p$ is $x_0 = 0$.  So in the coordinates described above, we have $a_{0,0,d} = a_{0,1,d-1} = 0$.

Now consider the 1-parameter subgroup with weights $(-1,0,1)$.  We have
$$ \mu_{\lambda} (x_2 , a_{i,j,k}) = d-2 +  i - k$$
which is negative whenever $i-k-2 < -d = -i-j-k$, or $2i + j < 2$.  This only occurs when both $i = 0$ and $j < 2$, in other words, when either $a_{0,0,d}$ or $a_{0,1,d-1}$ is nonzero.  By assumption, however, this is not the case, so $(p,C) \notin X^s (d-2)$.  Since $(p,C)$ was arbitrary, it follows that $X^s (d-2) = \emptyset$.

\end{proof}

Next, we identify the adjacent chamber in the $G$-ample cone.  It lies between the slopes corresponding to the Wronskians $W_1$ and $W_2$.  In what follows, we let $S$ denote the set of all pointed curves $(p,C)$ admitting the following description:  $C$ consists of a smooth conic together with $d-2$ copies of the tangent line through a point $q \neq p$ on $C$.  Notice that $S \subset H_2 '$.

\begin{proposition}

For any $t \in (d- \frac{9}{4} , d-2)$, $X^s (t) = X^{ss} (t) = X \backslash (H_1 \cup S)$.

\end{proposition}

\begin{proof}

We first show that $X^{ss} (t) \subseteq X \backslash H_1$.  Suppose that $(p,C) \in H_1$.  As before, by change of coordinates, we may assume that $p = (0,0,1)$ and the tangent line to $C$ at $p$ is $x_0 = 0$.  Since $(p,C) \in H_1$, either $p$ is a singular point of $C$ or this tangent line intersects $C$ at $p$ with multiplicity at least 3.  Thus we have $a_{0,0,d} = a_{0,1,d-1} = 0$, and either $a_{1,0,d-1} = 0$ (if $p$ is singular) or $a_{0,2,d-2} = 0$ (if $p$ is a flex).

We first examine the case where $p$ is a flex.  In this case, consider the 1-parameter subgroup with weights $(-5,1,4)$.  Then
$$ \mu_{\lambda} (x_2 , a_{i,j,k}) = 4t + 5i - j - 4k > 4d - 9 + 5i - j - 4k = 9i + 3j - 9$$
which is non-negative when $3i + j \geq 3$.  Since, by assumption, $C$ has no non-zero terms with both $i = 0$ and $j < 3$, we see that $(p,C) \notin X^{ss} (t)$.

Next we look at the case where $p$ is a singular point.  Consider the 1-parameter subgroup with weights $(-1,-1,2)$.  Then we have
$$ \mu_{\lambda} (x_2 , a_{i,j,k}) = 2t + i + j - 2k > 2d- \frac{9}{2} + i + j - 2k = 3i + 3j - \frac{9}{2}$$
which is non-negative when $i + j \leq \frac{3}{2}$.  By assumption, $C$ has no non-zero terms where one of $i,j$ is 0 and the other is at most 1, so $(p,C) \notin X^{ss} (t)$.  It follows that $X^{ss} (t) \subseteq X \backslash H_1$.

Next we show that $X^{ss} (t) \subseteq X \backslash S$.  Suppose that $(p,C) \in S$.  Without loss of generality, we may assume that $C$ is of the form
$$ C = x_0^{d-2} ( a_{d,0,0} x_0^2 + a_{d-1,1,0} x_0 x_1 + a_{d-2,2,0} x_1^2 + a_{d-1,0,1} x_0 x_2 ).$$
Now, consider the 1-parameter subgroup with weights $(-1,0,1)$.  Then
$$ \mu_{\lambda} ( x_l , a_{i,j,k} ) \geq -t + i - j  > 2-d + i - j $$
which is non-negative when $i - j \geq d-2$.  It follows that $(p,C) \notin X^{ss} (t)$.

Now we show that $X \backslash (H_1 \cup S) \subseteq X^s (t)$.  Suppose that $(p,C) \notin X^s (t)$.  Then there is a nontrivial 1-parameter subgroup that acts on $(p,C)$ with strictly positive weight.  By change of basis, we may assume that this subgroup acts with weights $(r_0 , r_1 , r_2 )$, with $r_0 \leq r_1 \leq r_2$.  Since this is a nontrivial subgroup of $PSL(3, \mathbb{C})$, we know that $r_0 < 0 < r_2$ and $r_0 + r_1 + r_2 = 0$.  We then have
$$ \mu_{\lambda} (x_l , a_{i,j,k} ) = tr_l - (r_0 i + r_1 j + r_2 k) > 0$$
We divide this into three cases, depending on $p$.

\textbf{Case 1 -- $p = (0,0,1)$:}  In this case, $r_l = r_2$.  If $r_1 \geq 0$, then $t r_2 < (d - 2)r_2 \leq 2r_1 + (d-2)r_2$.  On the other hand, if $r_1 < 0$< then $t r_2 < (d - 2)r_2 < r_0 + (d-1)r_2$.  Since the subgroup acts with strictly positive wieght, it follows that $a_{0,0,d} = a_{0,1,d-1} = 0$, and either $a_{1,0,d-1} = 0$ or $a_{0,2,d-2} = 0$.  Hence, $(p,C) \in H_1$.

\textbf{Case 2 -- $p$ lies on the line $x_0 = 0$, but not on the line $x_1 = 0$:}  In this case, $r_l = r_1$.  If $r_1 > 0$, then since $r_1 \leq r_2$, we have $t r_1 < dr_1 \leq r_1 j + r_2 (d-j)$, so we see that $a_{0,0,d} = a_{0,1,d-1} = \cdots = a_{0,d,0} = 0$.  This means that $p$ lies on a linear component of $C$, and therefore $(p,C) \in H_1$.

On the other hand, if $r_1 \leq 0$, then since $r_2 \geq -2 r_1$, we see that $t r_1 \leq (d-3)r_1 \leq (d-1)r_1 + r_2 \leq r_1 j + (d-j) r_2 + r_2$ for $j \leq d-1$, so $a_{0,0,d} = a_{0,1,d-1} = \cdots = a_{0,d-1,1} = 0$.  This means that either $p$ lies on a linear component of $C$ or the only point of $C$ lying on the line $x_0 = 0$ also lies on the line $x_1 =0$.  Again, we see that $(p,C) \in H_1$.

\textbf{Case 3 -- $p$ does not lie on the line $x_0 = 0$:}  In this case, $r_l = r_0$.  Since $r_0 < 0$ and $r_1 < r_0 < r_2$, we see that $tr_0 < (d-3)r_0 = (d-2)r_0 + r_1 + r_2 < r_0 i + r_1 j + r_2 k$ for $i \leq d-2, k \neq 0$.  Now, if $r_0 \geq 4r_1$, then we have $tr_0 < (d - \frac{9}{4} ) r_0 = (d- \frac{5}{4} )r_0 + r_1 + r_2 \leq (d-1)r_0 + r_2$.  It follows that $C$ is of the form
$$ C = \sum_{i+j=d} a_{i,j,0} x_0^i x_1^j.$$
In other words, $C$ is a union of $d$ lines.  In this case, the tangent line to every point of $C$ is a component of $C$ itself, so $(p,C) \in H_1$.

On the other hand, if $r_0 < 4r_1$, then $tr_0 < (d- \frac{9}{4} ) r_0 = (d-3)r_0 + \frac{3}{4} r_0 < (d-3)r_0 + 3 r_1$.  It follows that $C$ is of the form
$$ C = x_0^{d-2} ( a_{d,0,0} x_0^2 + a_{d-1,1,0} x_0 x_1 + a_{d-2,2,0} x_1^2 + a_{d-1,0,1} x_0 x_2 )$$
hence $C \in S$.

\end{proof}

We now consider the wall in the $G$-ample cone determined by the Wronskian $W_2$.

\begin{proposition}

A wall of the $G$-ample cone occurs at $t = d- \frac{9}{4}$.  More specifically, $X^{ss} (t) = X \backslash ((H_1 \cap H_2 ') \cup S)$, and $X^s (t) \subseteq X \backslash (H_1 \cup S)$.

\end{proposition}

\begin{proof}

First, notice that if $(p,C) \notin H_2 '$, then $(p,C) \in X^{ss} (t)$, since $W_2 '$ is a $G$-invariant section of $\OO_{X} (12(d- \frac{9}{4}), 12)$ that does not vanish at $(p,C)$.  Moreover, by general variation of GIT we know that, when passing from a chamber to a wall, we have
$$X^{ss} (t + \epsilon ) \subseteq X^{ss} (t)$$
$$X^s (t) \subseteq X^s (t + \epsilon )$$
Thus, $X^s (t) \subseteq X \backslash (H_1 \cup S)$ and $X \backslash ((H_1 \cap H_2 ') \cup S) \subseteq X^{ss} (t)$.

Now, suppose that $(p,C) \in S$.  Using the same argument as above with the same 1-parameter subgroup, we see that $(p,C) \notin X^{ss} (t)$.

Next, suppose that $(p,C) \in H_1$.  If $p$ is a singular point of $C$, then we see that $(p,C) \notin X^{ss} (t)$ by the same argument as before, using the subgroup with weights $(-1,-1,2)$.

The only other possibility is that $p$ is a flex.  In this case, we again consider the 1-parameter subgroup with weights $(-5,1,4)$.  As before, we have
$$ \mu_{\lambda} (x_2 , a_{i,j,k}) = 4d - 9 + 5i - j - 4k = 9i + 3j - 9$$
which is non-negative when $3i + j \geq 3$.  As before, we see that $(p,C) \notin X^s (t)$.

Notice furthermore that if the tangent line to $C$ at $p$ is a hyperflex, then $a_{0,3,d-3} = 0$ as well, and so the expression $3i + j - 3$ above is strictly positive (rather than simply nonnegative), and thus $(p,C) \notin X^{ss} (t)$.  If $(p,C) \in H_1 \cap H_2 '$, then either $p$ is a singular point of $C$, or $C$ is smooth at $p$ and the tangent line to $C$ at $p$ is a hyperflex.  From our observations above, we may therefore conclude that $X^{ss} (t) \subseteq X \backslash ((H_1 \cap H_2 ') \cup S)$.

\end{proof}

We are left to consider the behavior of our quotient at the wall crossing defined by $t_0 = d-\frac{9}{4}$.  As in Theorem \ref{Thaddeus}, we let
$$X^{\pm} = X^{ss} (t_0) \backslash X^{ss} (t_0 \mp \epsilon )$$
$$X^0 = X^{ss} (t_0 ) \backslash (X^{ss} (t_0 + \epsilon ) \cup X^{ss} (t_0 - \epsilon ))$$
Our first task is to determine $X^-$ and $X^0$ in this situation.

\begin{proposition}

With the set-up above, $X^- = H_1 \backslash H_2 '$.  $X^0$ is the set of all pointed curves $(p,C)$ consisting of a cuspidal cubic plus $d-3$ copies of the projectivized tangent cone at the cusp.  The point $p$ is the unique smooth flex point of the cuspidal cubic.

\end{proposition}

\begin{proof}

We have already seen that $X^{ss} (t_0 ) = X \backslash ((H_1 \cap H_2 ') \cup S)$ and $X^{ss} (t_0 + \epsilon ) = X \backslash (H_1 \cup S)$.  Thus, $X^- = H_1 \backslash H_2 '$.

To prove the statement about $X^0$, let $(p,C) \in X^0$.  Notice that, since $X^0 \subseteq X^-$, $p$ is a smooth point of $C$ and the tangent line to $C$ at $p$ intersects $C$ with multiplicity exactly 3.  Since $(p,C) \notin X^{ss} (t_0 - \epsilon )$, there must be a nontrivial 1-parameter subgroup that acts on $(p,C)$ with strictly positive weight.  Again we assume that this subgroup acts with weights $(r_0 , r_1 , r_2 )$, with $r_0 \leq r_1 \leq r_2$.  As before, we know that $r_0 < 0 < r_2$ and $r_0 + r_1 + r_2 = 0$.  Again we have
$$ \mu_{\lambda} (x_l , a_{i,j,k} ) = tr_l - (r_0 i + r_1 j + r_2 k) > 0$$
We divide this into three cases, depending on $p$.

\textbf{Case 1 -- $p = (0,0,1)$:}  In this case, $r_l = r_2$.  Now, if $t r_2 \geq r_0 + (d-1)r_2$, then $(d - \frac{9}{4})r_2 > r_0 + (d-1)r_2$, so $r_1 > \frac{1}{4} r_2$. This means that $t r_2 < (d - \frac{9}{4})r_2 < 3r_1 + (d-3)r_2$.  It follows that $a_{0,0,d} = a_{0,1,d-1} = 0$, and either $a_{1,0,d-1} = 0$ or $a_{0,2,d-2} = a_{0,3,d-3} = 0$.  But we know that $p$ is a smooth point of $C$ and the tangent line to $C$ at $p$ intersects $C$ with multiplicity exactly 3, so neither of these is a possibility.

\textbf{Case 2 -- $p$ lies on the line $x_0 = 0$, but not on the line $x_1 = 0$:}  Using the same argument as before, we see that $p$ lies on a linear component of $C$, which is impossible.

\textbf{Case 3 -- $p$ does not lie on the line $x_0 = 0$:}  In this case, $r_l = r_0$.  Again, since $r_0 < 0$ and $r_1 < r_0 < r_2$, we see that $tr_0 < (d-3)r_0 = (d-2)r_0 + r_1 + r_2 < r_0 i + r_1 j + r_2 k$ for $i \leq d-2, k \neq 0$.  Notice that, if $tr_0 < (d-1)r_0 + r_2$, then as before we see that $C$ is the union of $d$ lines, which is impossible.

We therefore see that $(d - \frac{12}{5})r_0 > t r_0 \geq (d-1)r_0 + r_2$.  But then $ \frac{7}{5}r_0 < -r_2 = r_0 + r_1$, so $r_0 < \frac{5}{2} r_1$.  It follows that $t r_0 < ( d- \frac{12}{5} )r_0 < (d-4)r_0 + 4r_1 \leq r_0 i + r_1 j$ for $j \geq 4$.

We see that $C$ is of the form
$$ C = x_0^{d-3} (a_{d,0,0}x_0^3 + a_{d-1,1,0}x_0^2 x_1 + a_{d-2,2,0}x_0 x_1^2 + a_{d-3,3,0}x_1^3 + a_{d-1,0,1}x_0^2 x_2). $$
Thus, $C$ consists of a cuspidal cubic together with $d-3$ copies of the projectivized tangent cone to the cusp.  The point $p$ is the unique flex point of the cuspidal cubic.

It is clear that this $(p,C) \in X^{-}$, since the tangent line to $C$ at $p$ intersects $C$ with multiplicity exactly 3.  To see that $(p,C) \notin X^{ss}(t_0 - \epsilon )$, consider again the 1-parameter subgroup with weights $(5,-1,-4)$.  The statement then follows from the fact that all cuspidal plane cubics are projectively equivalent.

\end{proof}

\begin{corollary}

The map $X //_{t_0 - \epsilon} G \to X //_{t_0} G$ contracts the locus $H_1 \backslash H_2 '$ to a point.  Outside of this locus, the map is an isomorphism.

\end{corollary}

\begin{proof}

Let $(p,C) \in X^0$.  Since all cuspidal plane cubics are projectively equivalent, $G \cdot (p,C) = X^0$, so $G \cdot (p,C)$ is closed in $X^{ss}(t_0 )$ and $X^0 // G$ is a point.  An automorphism of $\PP^1$ extends to $(p,C)$ if and only if it fixes the point $p$ and the cusp, and thus the stabilizer of $(p,C)$ is isomorphic to $\mathbb{C}^*$.  The conclusion follows from Theorem \ref{Thaddeus}.

\end{proof}

We are particularly interested in the case where $d=4$, because in this case $X //_{t_0 - \epsilon} G$ is a birational model for $\overline{M}_{3,1}$.  In particular, we have the following:

\begin{proposition}

There is a birational contraction $\beta : \overline{M}_{3,1} \dashrightarrow X //_{t_0 - \epsilon} G$.

\end{proposition}

\begin{proof}

It suffices to exhibit a morphism $\beta^{-1}: V \to \overline{M}_{3,1}$, where $V \subseteq X //_{t_0 - \epsilon} G$ is open with complement of codimension $\geq 2$ and $\beta^{-1}$ is an isomorphism onto its image.  To see this, let $U \subseteq X^{ss}(t_0 - \epsilon )$ be the set of all moduli stable pointed curves $(p,C) \in X^{ss}(t_0 - \epsilon )$.  Notice that the complement of $U$ is strictly contained in the discriminant locus $\Delta$, which is an irreducible hypersurface in $X$.  Note furthermore that there are stable points contained in both $X \backslash \Delta$ and $\Delta \cap U$.  Thus, the containments $(X \backslash U) //_{t_0 - \epsilon} G \subset \Delta //_{t_0 - \epsilon} G$ and $\Delta //_{t_0 - \epsilon} G \subset X //_{t_0 - \epsilon} G$ are strict.  It follows that the complement of $U // G$ in the quotient has codimension $\geq 2$.

By the universal property of the moduli space, since $U \to Z$ is a family of moduli stable curves, it admits a unique map $U \to Z \to \overline{M}_{3,1}$.  This map is certainly $G$-equivariant, so it factors uniquely through a map $U //_{t_0 - \epsilon} G \to \overline{M}_{3,1}$.  Since every degree 4 plane curve is canonical, two such curves are isomorphic if and only if they differ by an automorphism of $\PP^2$.  It follows that this map is an isomorphism onto its image.

\end{proof}

\begin{theorem}
There is a birational contraction of $\overline{M}_{3,1}$ contracting the Weierstrass divisor $BN^1_{3,(0,3)}$.  Furthermore, the divisors $BN^1_{3,(0,3)}$, $BN^1_2$, $\Delta_1$ and $\Delta_2$ span a simplicial face of $\overline{NE}^1 ( \overline{M}_{3,1} )$.
\end{theorem}

\begin{proof}

The composition $\overline{M}_{3,1} \dashrightarrow X //_{t_0 - \epsilon} G \to X //_{t_0} G$ is a birational contraction.  By the above, the Weierstrass divisor is contracted by this map, so it suffices to show that $BN^1_2$ and the $\Delta_i$'s are contracted as well.  We first note that every smooth curve in $X$ is canonically embedded and hence non-hyperelliptic, so the closure of the image of each of these divisors must be contained in the singular locus $\Delta$, which is an irreducible hypersurface.  Since the generic point of $\Delta$ is an irreducible nodal curve, we see that the closure of the image of $\Delta_i$ is of codimension 2 or greater for all $i \geq1$.  Moreover, since the set of singular hyperelliptic curves has codimension 2 in $\overline{M}_{3,1}$, and we have constructed an embedding of an open subset of $\Delta$ into $\overline{M}_{3,1}$, we may conclude that $BN^1_2$ is contracted as well.

\end{proof}

\section{Contraction of $\overline{M}_{4,1}$}

We now turn to the case of genus 4 curves.  Our main result will be the following:

\begin{theorem}
There is a birational contraction of $\overline{M}_{4,1}$ contracting the pointed Brill-Noether divisor $BN^1_{3,(0,2)}$.
\end{theorem}

In a similar way to the previous section, we will construct a birational model for $\overline{M}_{4,1}$ by considering GIT quotients of the universal family over the space of curves in $\PP^1 \times \PP^1$.  Here, the Hessian $H_{0,1}$ is again the image of a pointed Brill-Noether divisor.  As above, our goal is to find a GIT quotient in which this locus is contracted.  Let $Y = \PP^1 \times \PP^1$ and
$$X = \{ (p,C) \in Y \times \vert \OO (d,d) \vert \text{  } \vert p \in C \}.$$
Then $\pi_2 : X \to \vert \OO (d,d) \vert$ is the family of all curves of bidegree $(d,d)$.  Our goal, as before, is to study the GIT quotients of $X$ by the action of $G = PSO(4, \mathbb{C} )$.  By the above, we know that $PicX \cong \Z^3$, but we are only interested in those line bundles of the form $\OO_X (a,a,b)$.  We can therefore define the slope of a line bundle $\LL \in PicX$ as above.

\begin{definition}

We say a line bundle $\LL$ has \textbf{slope} $t$ if $\LL = \pi_1^* \OO (a,a) \otimes \pi_2^* \OO (b)$ with $t = \frac{a}{b}$.  We write $X^s (t)$ and $X^{ss} (t)$ for the sets of stable and semistable points, and $X //_t G $ for the corresponding GIT quotient.

\end{definition}

Here we describe the numerical criterion for points in $X$.  Let $p = (x_0 , x_1 : y_0 , y_1 )$ and
$$ C = \sum_{i_0 + i_1 = j_0 + j_1 = d} a_{i_0 ,i_1 ,j_0 ,j_1} x_0^{i_0} x_1^{i_1} y_0^{j_0} y_1^{j_1}.$$
Then a basis for $H^0 ( \OO_X (a,a,b))$ consists of monomials of the form
$$ \prod_{\alpha_0 = 1}^a x_{l_{ \alpha_0 }} y_{m_{ \alpha_1 }} \prod_{\beta = 1}^b a_{i_{0 \beta},i_{1 \beta} , j_{0 \beta}, j_{1 \beta}}. $$
The one-parameter subgroup with weights $(-r_0 , r_0, -r_1 , r_1 )$ acts on the monomial above with weight
$$ \sum_{\beta = 1}^b ( r_0 (i_{0 \beta } - i_{1 \beta }) + r_1 (j_{0 \beta } - j_{1 \beta }) ) - \sum_{\alpha_0 =1}^a ((-1)^{l_{\alpha_0}} r_0 + (-1)^{m_{\alpha_1}} r_1 ) . $$
In our case, we will only be interested in maximizing or minimizing this weight, so it suffices to consider monomials of the form $x_l^a y_m^a a_{i_0 , i_1 , j_0 , j_1}^b$.  In this case, the one-parameter subgroup acts with weight $b(r_0 (i_0 - i_1 ) + r_1 (j_0 - j_1 )) - a((-1)^l r_0 + (-1)^m r_1)$, which is proportional to
$$ \mu_{\lambda} (x_l , y_m , a_{i_0 , i_1 , j_0 , j_1} ) := r_0 (i_0 - i_1 ) + r_1 (j_0 - j_1 ) - t((-1)^l r_0 + (-1)^m r_1). $$

As in the previous section, when $t = 0$, we obtain a moduli space of curves of bidegree $(d,d)$.  In particular, the case $d=3$ is notable for being a birational model for $\overline{M}_4$.  We will see that as $t$ increases, stable curves are allowed to have more complicated singularities, but the vanishing sequences of linear series at the marked point become more well-controlled.  We begin by identifying an edge of the $G$-ample cone corresponding to the Wronskian $W_{0,1}$.

\begin{proposition}

An edge of the $G$-ample cone occurs at $t=d-1$.

\end{proposition}

\begin{proof}

It suffices to show that $X^{ss}(d-1 ) \neq X^s (d-1) = \emptyset$.  It is clear that $X^{ss}(d-1) \neq \emptyset$, since $W_{0,1}$ is a $G$-invariant section of $\OO_X (2(d-1), 2(d-1), 2)$.

To show that $X^s (d-1) = \emptyset$, we invoke the numerical criterion.  Let $(p,C) \in X$.  By change of coordinates, we may assume that $p = (0,1:0,1)$.  So, in the coordinates described above, we have $a_{0,d,0,d} = 0$.

Now consider the 1-parameter subgroup with weights $(-1,1,-1,1)$.  We have
$$ \mu_{\lambda} ( x_1 , y_1 , a_{i_0 , i_1 , j_0 , j_1} ) = 2(d-1) + (i_0 - i_1 ) + (j_0 - j_1 ) $$
which is negative whenever $(i_0 - i_1 ) + (j_0 - j_1 ) < -2(d-1) = -i_0 - i_1 - j_0 - j_1 + 2$, or $i_0 + j_0 < 1$.  This only occurs when $i_0 = j_0 = 0$, in other words, when $a_{0,d,0,d}$ is nonzero.  By assumption, however, this is not the case, so $(p,C) \notin X^s (d-1)$.  Since $(p,C)$ was arbitrary, it follows that $X^s (d-1) = \emptyset$.

\end{proof}

As above, we identify the adjacent chamber in the $G$-ample cone.  It lies between the slopes corresponding to the Wronskians $W_{0,1}$ and $W_{1,1}$.  In what follows, we let $S$ denote the set of all pointed curves $(p,C)$ admitting the following description:  $C$ consists of a smooth curve of bidegree $(1,1)$ together with $d-1$ copies of the two lines through a point $q \neq p$ on $C$.  Notice that $S \subset H_{1,1}$.

\begin{proposition}

For any $t \in (d- \frac{4}{3} , d-1)$, $X^s (t) = X^{ss} (t) = X \backslash ( H_{0,1} \cup S)$.

\end{proposition}

\begin{proof}

We first show that $X^{ss} (t) \subseteq X \backslash H_{0,1}$.  Suppose that $(p,C) \in H_{0,1}$.  As before, by change of coordinates, we may assume that $p = (0,1:0,1)$.  Since $(p,C) \in H_{0,1}$, $C$ intersects one of the two lines through $p$ with multiplicity at least 2.  Without loss of generality, we may assume this line to be $x_0 = 0$.  Thus, if we write $C$ as above, then $a_{0,d,0,d} = a_{0,d,1,d-1} = 0$.  Now, consider the 1-parameter subgroup with weights $(-1,1,-2,2)$.  Then
$$ \mu_{\lambda} ( x_1 , y_1 , a_{i_0 , i_1 , j_0 , j_1} ) = 3t + i_0 - i_1 + 2j_0 - 2j_1 > 3d-4 + i_0 - i_1 + 2j_0 - 2j_1$$
$$ = 2(i_0 + 2j_0 - 2)$$
which is non-negative when $i_0 + 2j_0 \geq 2$.  Since, by assumption, $C$ has no non-zero terms with both $i_0 \leq 1$ and $j = 0$, we see that $(p,C) \notin X^{ss} (t)$.

Next we show that $X^{ss} (t) \subseteq X \backslash S$.  Suppose that $(p,C) \in S$.  Without loss of generality, we may assume that $C$ is of the form
$$ C = x_0^{d-1} y_0^{d-1} ( a_{d,0,d,0} x_0 y_0 + a_{d-1,1,d,0} x_1 y_0 + a_{d,0,d-1,1} x_0 y_1 ).$$
Now, consider the 1-parameter subgroup with weights $(1,-1,1,-1)$.  Then
$$ \mu_{\lambda} ( x_l , y_m , a_{i_0 , i_1 , j_0 , j_1} ) \geq -2t - i_0 + i_1 - j_0 + j_1 > -2d+2 - i_0 + i_1 - j_0 + j_1 $$
$$ = -2(i_1 + j_1 - 1)$$
which is non-negative when $i_1 + j_1 \leq 1$.  It follows that $(p,C) \notin X^{ss} (t)$.

Now we show that $X \backslash (H_{0,1} \cup S) \subseteq X^s (t)$.  Suppose that $(p,C) \notin X^s (t)$.  Then there is a nontrivial 1-parameter subgroup that acts on $(p,C)$ with strictly positive weight.  By change of basis, we may assume that this subgroup acts with weights $(-r_0 , r_0 , -r_1, r_1 )$, with $0 \leq r_0 \leq r_1$ and $r_1 > 0$.  We then have
$$ \mu_{\lambda} (x_l , y_m , a_{i_0 , i_1 , j_0 , j_1} ) = r_0 (i_0 - i_1 ) + r_1 (j_0 - j_1 ) - t((-1)^l r_0 + (-1)^m r_1) > 0 $$
We divide this into four cases, depending on $p$.

\textbf{Case 1 -- $p = (0,1:0,1)$:}  In this case, $l = m = 1$.  We have $t(-r_0 - r_1 ) > (d-1)(-r_0 - r_1 ) \geq -(d-2)r_0 - dr_1$.  It follows that $a_{0,d,0,d} = a_{1,d-1,0,d} = 0$, so $(p,C) \in H_{0,1}$.

\textbf{Case 2 -- $p$ lies on the line $y_0 = 0$, but not the line $x_0 = 0$:}  In this case, $l = 1$ and $m = 0$.  Here, $t(-r_0 + r_1 ) > (d-2)(-r_0 + r_1 ) \geq -dr_0 + kr_1$ for all $k \leq d-2$.  We therefore see that $a_{0,d,k,d-k} = 0$ for all $k \leq d-2$.  If $a_{0,d,d,0} \neq 0$, then every point of $C$ that lies on the line $x_0 = 0$ also lies on the line $y_0 = 0$, a contradiction.  We therefore see that $a_{0,d,d,0} = 0$ as well, but this means that $p$ lies on a linear component of $C$, and therefore $(p,C) \in H_{0,1}$.

\textbf{Case 3 -- $p$ lies on the line $x_0 = 0$, but not on the line $y_0 = 0$:}  In this case, $l = 0$ and $m = 1$.  Note that $t(r_0 - r_1 ) > d(-r_0 + r_1 ) \geq -dr_0 + kr_1$ for all $k \leq d$.  It follows that $a_{k,d-k,0,d} = 0$ for all values of $k$, which means that $y_0 = 0$ is a linear component of $C$.  Thus $(p,C) \in H_{0,1}$.

\textbf{Case 4 -- $p$ does not lie on either of the lines $x_0 = 0$ or $y_0 = 0$:}  In this case, $l = m = 0$.  Now note that $t(r_0 + r_1 ) > (d-2)(r_0 + r_1 )$, so $a_{k_0 ,d-k_0 ,k_1 ,d-k_1} = 0$ if $k_0$ and $k_1$ are both less than $d$.  Furthermore, since $r_0 \leq r_1$, $(d-2)(r_0 + r_1 ) \geq dr_0 + (d-4)r_1$, so $a_{d,0,k,d-k} = 0$ for $k \leq d-2$.  Now, if $(d- \frac{4}{3} ) (r_0 + r_1 ) \leq (d-4)r_0 + dr_1$, then $2r_0 \leq r_1$, so $t(r_0 + r_1 ) > (d- \frac{4}{3} )(r_0 + r_1 ) \geq dr_0 + (d-2)r_1$.  It follows that either $a_{d,0,d-1,1} = 0$, in which case $C$ is a union of $2d$ lines and hence $(p,C) \in H_{0,1}$, or $a_{k,d-k,d,0} = 0$ for all $k \leq d-2$, in which case $C \in S$.

\end{proof}

We now consider the wall in the $G$-ample cone determined by the Wronskian $W_{1,1}$.

\begin{proposition}

A wall of the $G$-ample cone occurs at $t = d- \frac{4}{3}$.  More specifically, $X^{ss} (t) = X \backslash ((H_{0,1} \cap H_{1,1}) \cup S)$, and $X^s (t) \subseteq X \backslash (H_{0,1} \cup S)$.

\end{proposition}

\begin{proof}

First, notice that if $(p,C) \notin H_{1,1}$, then $(p,C) \in X^{ss} (t)$, since $W_{1,1}$ is a $G$-invariant section of $\OO_X (6(d- \frac{4}{3} ), 6(d- \frac{4}{3} ), 6)$ that does not vanish at $(p,C)$.   Moreover, by general variation of GIT we know that, when passing from a chamber to a wall, we have
$$X^{ss} (t + \epsilon ) \subseteq X^{ss} (t)$$
$$X^s (t) \subseteq X^s (t + \epsilon )$$
Thus, $X^s (t) \subseteq X \backslash (H_{0,1} \cup S)$ and $X \backslash ((H_{0,1} \cap H_{1,1}) \cup S) = X^{ss} (t)$.

Now, suppose that $(p,C) \in S$.  Using the same argument as before with the same 1-parameter subgroup, we see that $(p,C) \notin X^{ss} (t)$.

Next, suppose that $(p,C) \in H_{0,1}$.  In this case, we again consider the 1-parameter subgroup with weights $(-1,1,-2,2)$.  As before, we have
$$ \mu_{\lambda} ( x_1 , y_1 , a_{i_0 , i_1 , j_0 , j_1} ) = 3d-4 + i_0 - i_1 + 2j_0 - 2j_1 = 2(i_0 + 2j_0 - 2)$$
which is non-negative when $i_0 + 2j_0 \geq 2$.  Since, by assumption, $C$ has no non-zero terms with both $i_0 \leq 1$ and $j = 0$, we see that $(p,C) \notin X^s (t)$.

Notice furthermore that if $(p,C) \in H_{0,1} \cap H_{1,1}$, this means that the osculating $(1,1)$ curve to $C$ at $p$ is the pair of lines through that point, and this curve intersects $C$ with multiplicity at least 4.  This means that either $a_{0,d,1,d-1} = 0$ or $a_{2,d-2,0,d} = 0$, which implies that the expression $i_0 + 2j_0 - 2$ above is zero for at most one term, and strictly positive for all of the others.  Now consider the 1-parameter subgroup with weights $(-1 - \epsilon , 1 + \epsilon , -2,2)$.  For $\epsilon > 0$, we see that any curve with $a_{0,d,1,d-1} = 0$ is unstable.  Conversely, if $\epsilon < 0$, we see that any curve with $a_{2,d-2,0,d} = 0$ is unstable.  It follows that $(p,C) \notin X^{ss} (t)$, and thus $X^{ss} (t) = X \backslash ((H_{0,1} \cap H_{1,1}) \cup S)$.

\end{proof}

Again, we want to use Theorem \ref{Thaddeus} to study the wall crossing at $t_0 = d- \frac{4}{3}$.  Again, we let
$$X^{\pm} = X^{ss} (t_0) \backslash X^{ss} (t_0 \mp \epsilon )$$
$$X^0 = X^{ss} (t_0 ) \backslash (X^{ss} (t_0 + \epsilon ) \cup X^{ss} (t_0 - \epsilon ))$$
and determine $X^-$ and $X^0$.

\begin{proposition}

With the set-up above, $X^- = H_{0,1} \backslash H_{1,1}$.  $X^0$ is the set of all pointed curves $(p,C)$ admitting the following description:  $C$ consists of a smooth curve of bidegree $(1,2)$ (or $(2,1)$), together with $d-1$ copies of the tangent line to this curve through a point that has a tangent line, and $d-2$ copies of the other line through this same point.  The marked point $p$ is the unique other point on the smooth $(1,2)$ curve that has a tangent line.

\end{proposition}

\begin{proof}

We have already seen that $X^{ss} (t_0) = X \backslash ((H_{0,1} \cap H_{1,1}) \cup S)$ and $X^{ss} (t_0 + \epsilon ) = X \backslash (H_{0,1} \cup S)$.  Thus, $X^- = H_{0,1} \backslash H_{1,1}$.

To prove the statement about $X^0$, let $(p,C) \in X^0$.  Notice that, since $X^0 \subseteq X^-$, exactly one of the two lines through $p$ intersects $C$ with multiplicity exactly 2.  Since $(p,C) \notin X^{ss} (t_0 - \epsilon )$, there must be a nontrivial 1-parameter subgroup that acts on $(p,C)$ with strictly positive weight.  Again we assume that this subgroup acts with weights $(-r_0 , r_0 , -r_1, r_1 )$, with $0 \leq r_0 \leq r_1$ and $r_1 > 0$.  Again we have
$$ \mu_{\lambda} (x_l , y_m , a_{i_0 , i_1 , j_0 , j_1} ) = r_0 (i_0 - i_1 ) + r_1 (j_0 - j_1 ) - t((-1)^l r_0 + (-1)^m r_1) > 0 $$
We divide this into four cases, depending on $p$.

\textbf{Case 1 -- $p = (0,1:0,1)$:}  In this case, $l = m = 1$.  Again we have $t(-r_0 - r_1 ) > (d-1)(-r_0 - r_1 ) \geq -(d-2)r_0 - dr_1$.  Now, if $t(-r_0 - r_1 ) \leq -dr_0 - (d-2)r_1$, then $(d - \frac{4}{3})(-r_0 - r_1 ) < -dr_0 - (d-2)r_1$, so $r_1 > 2r_0$.  This means that $t(-r_0 - r_1 ) < (d - \frac{4}{3})(-r_0 - r_1 ) < -(d-4)r_0 - dr_1$.  It follows that $a_{0,d,0,d} = a_{1,d-1,0,d} = 0$, and either $a_{0,d,1,d-1} = 0$ or $a_{2,d-2,0,d} = 0$.  But we know that exactly one of the two lines through $p$ intersects $C$ with multiplicity exactly 2, so neither of these is a possibility.

\textbf{Case 2 -- $p$ lies on the line $y_0 = 0$, but not the line $x_0 = 0$:}  Following the same argument as above we see that either $p$ lies on a linear component of $C$, or every point of $C$ that lies on the line $x_0 = 0$ also lies on the line $y_0 = 0$.  It follows that $(p,C) \notin X^-$, a contradiction.

\textbf{Case 3 -- $p$ lies on the line $x_0 = 0$, but not on the line $y_0 = 0$:}  Again, following the same argument as above we see that $p$ lies on a linear component of $C$.  This implies that $(p,C) \notin X^-$, which is impossible.

\textbf{Case 4 -- $p$ does not lie on either of the lines $x_0 = 0$ or $y_0 = 0$:}  In this case, $l = m = 0$.  As above, we see that $a_{k_0 ,d-k_0 ,k_1 ,d-k_1} = 0$ if $k_0$ and $k_1$ are both less than $d$, and $a_{d,0,k,d-k} = 0$ for $k < d-1$.  Now, if $(d - \frac{3}{2})(r_0 + r_1 ) \leq (d-6)r_0 + dr_1$, then $3r_0 \leq r_1$, so $t(r_0 + r_1 ) > (d - \frac{3}{2})(r_0 + r_1 ) \geq dr_0 + (d-2)r_0$.  It follows that either $a_{d,0,d-1,1} = 0$, in which case $C$ is a union of $2d$ lines, which is impossible, or $a_{k,d-k,d,0} = 0$ for all $k < d-2$.  We therefore see that $C$ is of the form
$$ C = x_0^{d-2}y_0^{d-1} (a_{d,0,d,0}x_0^2 y_0 + a_{d,0,d-1,1}x_0^2 y_1 + a_{d-1,1,d,0}x_0 x_1 y_0 + a_{d-2,2,d,0}x_1^2 y_0 ). $$
Thus, $C$ consists of three components.  One is a curve of bidegree $(2,1)$.  The other two components consist of multiple lines through one of the points on this curve that has a tangent line.  The point $p$ is forced to be the unique other such point.

It is clear that this $(p,C) \in X^{-}$, since by definition, one of the lines through $p$ intersects $C$ with multiplicity greater than 1, and it is impossible for it to intersect a smooth curve of bidegree $(2,1)$ with higher multiplicity than 2, or for the other line through $p$ to intersect the curve with multiplicity at all.  To see that $(p,C) \notin X^{ss}(t_0 - \epsilon )$, consider the 1-parameter subgroup with weights $(-1,1,-2,2)$.

Finally, notice that all such curves are in the same orbit of the action of $G$, so $X^0$ must be the set of \emph{all} such curves.  To see this, note that if we fix the two points that have tangent lines to be $(1,0:1,0)$ and $(0,1:0,1)$, then the curve is determined uniquely by the third point of intersection of the curve with the diagonal.  Since $PSL(2, \mathbb{C} )$ acts 3-transitively on points of $\PP^1$, we obtain the desired result.

\end{proof}

\begin{corollary}

The map $X //_{t_0 - \epsilon} G \to X //_{t_0} G$ contracts the locus $H_{0,1} \backslash H_{1,1}$ to a point.  Outside of this locus, the map is an isomorphism.

\end{corollary}

\begin{proof}

Let $C = x_1^{d-2}y_1^{d-1}(x_0^2 y_1 + x_1^2 y_0 )$, and $p = (0,1:0,1)$.  Then $(p,C) \in X^0$.  As we have seen, $X^0$ is the orbit of $(p,C)$, so $G \cdot (p,C)$ is closed in $X^{ss}(t_0)$ and $X^0 //_{t_0} G$ is a point.  Notice that the stabilizer of $(p,C)$ must fix $p = (0,1:0,1)$, and the other ramification point, which is $(1,0:1,0)$.  Thus, the stabilizer of $(p,C)$ must consist solely of pairs of diagonal matrices.  A quick check shows that the stabilizer of $(p,C)$ is the one-parameter subgroup with weights $(-1,1,-2,2)$,
which is isomorphic to $\mathbb{C}^*$.  Again, the conclusion follows from Theorem \ref{Thaddeus}.

\end{proof}

Our main interest is the case where $d=3$.  As above, this is because in this case $X //_{t_0 - \epsilon} G$ is a birational model for $\overline{M}_{4,1}$.  In particular, we have the following:

\begin{proposition}

There is a birational contraction $\beta : \overline{M}_{4,1} \dashrightarrow X //_{t_0 - \epsilon} G$.

\end{proposition}

\begin{proof}

As above, it suffices to exhibit a morphism $\beta^{-1}: V \to \overline{M}_{4,1}$, where $V \subseteq X //_{t_0 - \epsilon} G$ is open with complement of codimension $\geq 2$ and $\beta^{-1}$ is an isomorphism onto its image.  Again, we let $U \subseteq X^{ss}(t_0 - \epsilon)$ be the set of all moduli stable pointed curves $(p,C) \in X^{ss}(t_0 - \epsilon)$.  The proof in this case is exactly like that in the case of $\PP^2$, as the discriminant locus $\Delta \subseteq X$ is again an irreducible $G$-invariant hypersurface.

By the universal property of the moduli space, since $U \to Z$ is a family of moduli stable curves, it admits a unique map $U \to Z \to \overline{M}_{4,1}$.  This map is certainly $G$-equivariant, so it factors uniquely through a map $U //_{t_0 - \epsilon} G \to \overline{M}_{4,1}$.  Since every curve of bidegree $(3,3)$ on $\PP^1 \times \PP^1$ is canonical, two such curves are isomorphic if and only if they differ by an automorphism of $\PP^1 \times \PP^1$.  It follows that this map is an isomorphism onto its image.

\end{proof}

\begin{theorem}
There is a birational contraction of $\overline{M}_{4,1}$ contracting the pointed Brill-Noether divisor $BN^1_{3,(0,2)}$.  Moreover, if $P$ is the Petri divisor, then the divisors $BN^1_{3,(0,2)}$, $P$, $\Delta_1$, $\Delta_2$, and $\Delta_3$ span a simplicial face of  $\overline{NE}^1 ( \overline{M}_{4,1} )$.
\end{theorem}

\begin{proof}

The composition $\overline{M}_{4,1} \dashrightarrow X //_{t_0 - \epsilon} G \to X //_{t_0} G$ is a birational contraction.  By the above, the given pointed Brill-Noether divisor is contracted by this map, so it suffices to show that $P$ and the $\Delta_i$'s are contracted as well.  The proof is again the same as the $\PP^2$ case.  We note that every smooth curve in $X$ is Petri general, so the closure of the image of each of these divisors must be contained in the singular locus $\Delta$, which is an irreducible hypersurface.  Since the generic point of $\Delta$ is an irreducible nodal curve, we see that the closure of the image of $\Delta_i$ is of codimension 2 or greater for all $i \geq1$.  Moreover, since the set of singular curves with semi-canonical pencils has codimension at least 2 in $\overline{M}_{4,1}$, and we have constructed an embedding of an open subset of $\Delta$ into $\overline{M}_{4,1}$, we may conclude that $P$ is contracted as well.

\end{proof}

\bibliography{ref}

\end{document}